\documentclass[12pt]{article}

\usepackage{amsmath}
\usepackage{amssymb}
\usepackage{amscd}
\usepackage{amsthm}
\theoremstyle{plain}
\newtheorem{Thm}{Theorem}

\newtheorem{Prop}[Thm]{Proposition}
\newtheorem{Conj}[Thm]{Conjecture}
\newtheorem{Subconj}{Subconjecture}
\newtheorem{Lemm}[Thm]{Lemma}
\newtheorem{Clai}{Claim}
\newtheorem{Cor}[Thm]{Corollary}
\theoremstyle{definition}
\newtheorem{Defn}[Thm]{Definition}
\newtheorem{Assum}[Thm]{Assumption}
\newtheorem{Proced}{Procedure}
\theoremstyle{remark}
\newtheorem{Rem}[Thm]{Remark}
\newtheorem{Exm}[Thm]{Example}

\theoremstyle{plain}
\newtheorem{ThmA}{Theorem A.}
\newtheorem{PropA}[ThmA]{Proposition A.}
\newtheorem{ConjA}[ThmA]{Conjecture A.}
\theoremstyle{definition}
\newtheorem{DefnA}[ThmA]{Definition A.}
\theoremstyle{remark}
\newtheorem{RemA}[ThmA]{Remark A.}

\numberwithin{equation}{section}
\def\Exc{\mathop{\mathrm{Exc}}}

\def\Bs{\mathop{\mathrm{Bs}}}
\def\Supp{\mathop{\mathrm{Supp}}}

\def\Diff{\mathop{\mathrm{Diff}}}

\begin{document}
\title{Note on quasi-numerically positive \\ log canonical divisors}
\author{ Shigetaka FUKUDA}
 
\date{\empty}

\maketitle \thispagestyle{empty}
\pagestyle{myheadings}
\markboth{Shigetaka FUKUDA}{Log canonical divisors}
\begin{abstract}
We propose a subconjecture that implies the semiampleness conjecture for quasi-numerically positive log canonical divisors and prove the ampleness in some elementary cases.
\end{abstract}

\section{Introduction}

In this note, every algebraic variety is defined over the field $\mathbf{C}$ of complex numbers.
We follow the terminology and notation in \cite{Ko}.

\begin{Defn}
Let $D$ be a $\mathbf{Q}$-Cartier $\mathbf{Q}$-divisor on a projective variety $X$.
The divisor $D$ is {\it numerically positive} ({\it nup}, for short), if $(D , C) > 0$ for every curve $C$ on $X$.
The divisor $D$ is {\it quasi-numerically positive} ({\it quasi-nup}, for short), if it is nef and if there exists a union $F$ of at most countably many prime divisors on $X$ such that $(D , C) > 0$ for every curve $C \nsubseteq F$ (i.e.\ if $D$ is nef and if $(D , C) > 0$ for every very general curve $C$).
\end{Defn}

\begin{Rem}
The quasi-nup divisors are the divisors ``of maximal nef dimension'' in the terminology of the ``Eight Authors'' \cite{BCEKPRSW}.
\end{Rem}

Ambro \cite{Am05} and Birkar-Cascini-Hacon-McKernan \cite{BCHM} reduced the famous log abundance conjecture to the termination conjecture for log flips and the semiampleness conjecture (Conjecture \ref{Conj}) for quasi-nup log canonical divisors $K_X + \Delta$, in the category of Kawamata log terminal ({\it klt}, for short) pairs.
In Section \ref{Sect:P} we propose a subconjecture (Subconjecture \ref{Subconj}) that implies the semiampleness Conjecture \ref{Conj}.

\begin{Rem}\label{Hist}
We state the history in detail (c.f.\ \cite{Fu14}).
In the category of klt pairs $(X, \Delta)$, Fukuda \cite{Fu02} (2002) reduced the log abundance to the existence and termination of log flips, the existence of log canonical bundle formula and the semiampleness of quasi-nup log canonical divisors, by using the numerically trivial fibrations (\cite{Tsuji}, see also \cite{BCEKPRSW}) due to Tsuji and the semiampleness criterion (\cite{Ka85a} and \cite{Ny86}, see also Fujino \cite{Fuj09}) for log canonical divisors due to Kawamata-Nakayama.
Ambro \cite{Am05} gave and proved the celebrated log canonical bundle formula.
The existence of log flips is now the theorem \cite{BCHM} due to Birkar, Cascini, Hacon and McKernan.
This history is along the line of Reid's philosophy stated in the famous Pagoda paper \cite{Re83}.

We also note two relevant theorems.
In Fukuda \cite{Fu1999} (Base point free theorem of Reid type, 1999), we proved that, if the log canonical divisor on a $\mathbf{Q}$-factorial divisorial log terminal variety is nef and log big, then it is semiample.
In Fukuda \cite{Fu11} (2011), we proved that, if the log canonical divisor on a klt variety is numerically equivalent to some semiample $\mathbf{Q}$-divisor, then it is semiample.
\end{Rem}

There is another approach to the semiampleness Conjecture \ref{Conj}.
Let $(X, \Delta)$ be a klt pair whose log canonical divisor $K_X + \Delta$ is quasi-nup.
Hacon and McKernan (Lazic \cite{La}, Theorem A.6) considered to embed $(X, \Delta)$ into some log canonical pair $(\overline{X}, \overline{\Delta})$ so that $\dim \overline{X} = \dim X + 1$ and $\overline{\Delta} \ge X$, that the log canonical divisor $K_{\overline{X}} + \overline{\Delta}$ is nef and big, that $(K_{\overline{X}} + \overline{\Delta}) \vert_X = K_X + \Delta$ and that $\overline{X}$ is endowed with the birational contraction morphism $\phi: \overline{X} \to \overline{Y}$ that contracts the prime divisor $X (= \Exc ( \phi ))$ to some point.
In Section \ref{Sect:E}, motivated by this consideration, we prove the ampleness (Theorem \ref{Thm:MT}) for log canonical pairs in some elementary cases.

In Appendix A, we survey the celebrated extension theorem (\cite{DHP}) which is recently proven by Demailly-Hacon-P\u aun.

In Appendix B, we give a straightforward proof to the theorem due to Boucksom-Cacciola-Lopez (\cite{BCL}) and Birkar-Hu (\cite{BH})-Cacciola (\cite{Ca}) that, for every divisorial log terminal pair whose log canonical divisor is strongly log big, the log canonical ring is finitely generated.

\section{Subconjecture for klt pairs}\label{Sect:P}

\begin{Conj}\label{Conj}
Let $(X,\Delta)$ be a Kawamata log terminal pair such that $X$ is projective.
If the log canonical divisor $K_X + \Delta$ is quasi-nup, then it is semiample.
\end{Conj}

We give an approach towards the above-mentioned semiampleness conjecture in this section.
The approach repeats the process of finding some $(K_X + \Delta)$-trivial curve that generates a $(K_X + \Gamma)$-extremal ray for some other klt pair $(X, \Gamma)$ and contracting this extremal ray.
The process would terminate at the ample log canonical divisor.
To run the process, it is important not to require the $\mathbf{Q}$-factoriality of $X$.

\begin{Defn}
We define $\overline{NE}_{D=0}(X):= \{ l \in \overline{NE}(X) \vert \text{the } \text{intersection } \text{number } (D,l)=0 \}$ and $\overline{NE}_{D \ge 0}(X):= \{ l \in \overline{NE}(X) \vert (D,l) \ge 0 \}$ for a $\mathbf{Q}$-Cartier $\mathbf{Q}$-divisor $D$ on $X$.
\end{Defn}

\begin{Subconj}\label{Subconj}
Let $(X,\Delta)$ be a Kawamata log terminal pair such that $X$ is projective.
Suppose that the log canonical divisor $K_X + \Delta$ is not ample but quasi-nup.
Then there exists an effective $\mathbf{Q}$-Cartier divisor $E$ such that the intersection number $(E,l) < 0$ for some class $l \in \overline{NE}_{K_X + \Delta =0} (X)$.
\end{Subconj}

\begin{Proced}\label{Proced}
Let $(X,\Delta)$ be a Kawamata log terminal pair such that $X$ is projective.
Suppose that the log canonical divisor $K_X + \Delta$ is not ample but quasi-nup.
Assume the existence of an effective $\mathbf{Q}$-Cartier divisor $E$ on $X$ and a member $l$ of $\overline{NE}_{K_X + \Delta =0} (X)$ such that the intersection number $(E,l) < 0$.
Let $\epsilon$ be a sufficiently small positive rational number.
We can write this class $l$ in the form that $l = l_0 + l_1 + l_2 + \cdots + l_p$ $(p \ge 1)$, where $l_0 \in \overline{NE}_{K_X + \Delta + \epsilon E \ge 0}(X)$ and $\mathbf{R}_+ l_i$ $(i \ge 1)$ are distinct $(K_X + \Delta + \epsilon E)$-extremal rays. 
Then $(K_X + \Delta, l_1) = 0$, because $K_X + \Delta$ is nef and $(K_X + \Delta, l) = 0$.
We consider the birational contraction morphism $\phi : X \to X_1$ of the $(K_X + \Delta + \epsilon E)$-extremal ray $\mathbf{R}_{+} l_1$.
Put $\Delta_1 : = \phi_* (\Delta)$.
We note that the Picard number $\rho (X_1) = \rho (X) - 1$, that $K_X + \Delta = \phi^* (K_{X_1} + \Delta_1)$, that $(X_1, \Delta_1)$ is Kawamata log terminal and that $K_{X_1} + \Delta_1$ is quasi-nup.
Remark that we can permit each of the divisorial-contraction case and the small-contraction case, because we do not require the $\mathbf{Q}$-factoriality of $X_1$.
\end{Proced}

Procedure \ref{Proced} relates Subconjecture \ref{Subconj} to Conjecture \ref{Conj}.
The following is the main result of this section:

\begin{Thm}
Subconjecture \ref{Subconj} implies Conjecture \ref{Conj}.
\end{Thm}

\begin{proof}
Let $(X, \Delta)$ be a Kawamata log terminal pair such that $X$ is projective and the log canonical divisor $K_X + \Delta$ is quasi-nup.
If Subconjecture \ref{Subconj} is true, then, by repeating Procedure \ref{Proced}, we obtain a Kawamata log terminal pair $(X', \Delta')$ with the birational morphism $\psi : X \to X'$ such that $K_X + \Delta = \psi^* (K_{X'} + \Delta')$ and that $K_{X'} + \Delta'$ is ample, because the Picard numbers decrease 1 by 1 in the process of contraction of extremal rays.
\end{proof}

\begin{Cor}
Subconjecture \ref{Subconj} and the termination conjecture for log flips imply the log abundance conjecture for klt pairs.
\end{Cor}

\begin{proof}
See Remark \ref{Hist} and the theorem above.
\end{proof}

\begin{Rem}
From the corollary above and the existence theorem \cite{Ka} for extremal rational curves by Kawamata, we can say that the log abundance conjecture is the existence problem for some kind of rational curves, modulo the termination of log flips.
\end{Rem}

We show that Subconjecture \ref{Subconj} is a part of Conjecture \ref{Conj}.

\begin{Lemm}\label{Lemm}
Let $(X, \Delta)$ be a Kawamata log terminal pair such that $X$ is projective.
Suppose that $K_X + \Delta$ is not ample but quasi-nup and semiample.
Then there exists an effective $\mathbf{Q}$-Cartier divisor $E$ such that the intersection number $(E,l) < 0$ for some class $l \in \overline{NE}_{K_X + \Delta =0} (X)$.
\end{Lemm}

\begin{proof}
Consider the surjective morphism $\phi :X \to Y ( = \Phi_{\vert k(K_X + \Delta) \vert}(X))$ induced by the linear system $\vert k(K_X + \Delta) \vert$ for a sufficiently large and divisible integer $k$.
This morphism $\phi$ becomes birational, because of the Stein factorisation
theorem and the fact that the pull-backs of ample divisors by finite
morphisms are ample.
Then $k(K_X + \Delta) = \phi^* H$ for an ample divisor $H$ on $Y$.
By the Kodaira Lemma, if $m$ is sufficiently large and divisible, then $m \phi^* H = A + E$ for some ample divisor $A$ and some effective divisor $E$.
For every $\phi$-exceptional curve $C$, we obtain the inequality that $(E, C) < 0$, because $(m \phi^* H, C)=0$ and $(A,C) >0$.
Here the class $[C]$ belongs to $\overline{NE}_{K_X + \Delta =0} (X)$.
\end{proof}

\begin{Prop}
Conjecture \ref{Conj} implies Subconjecture \ref{Subconj}.
\end{Prop}

\begin{proof}
Lemma \ref{Lemm} gives the assertion.
\end{proof}

\section{Log canonical pairs in some elementary cases}\label{Sect:E}

We prove the ampleness for log canonical pairs in some elementary cases.

\begin{Assum}\label{Assum:I}
Let $f: X \to Y$ be a birational morphism between normal projective varieties of dimension $n$ such that $E:= \Exc ( f )$ is a prime divisor and let $(X,\Delta)$ and $(X, \Delta + E)$ be divisorial log terminal pairs.
Assume that $K_X + \Delta + E$ is nup.
\end{Assum}

\begin{Prop}\label{Prop:I}
Under Assumption \ref{Assum:I}.
The divisor $K_X + \Delta + (1 - \epsilon ) E$ is nef for every small number $\epsilon > 0$.
\end{Prop}

\begin{proof}
The result \cite{Ka} of Kawamata for klt pairs and its variant (\cite{Sh}, Proposition 1) of Shokurov for dlt pairs give the boudedness of the length of $(K_X + \Delta)$-extremal rays.
By using the argument in \cite{KeMaMc}, we have $\nu := \inf \{ \frac{(K_X + \Delta + E, C)}{-(K_X + \Delta , C)} \vert $C$ \text{ is } \text{an } \text{extremal } \text{rational } \text{curve } \text{for } K_X + \Delta \} > 0 $.
Thus $K_X + \Delta + E + \nu (K_X + \Delta)$ is nef.
\end{proof}

\begin{Assum}\label{Assum:II}
Furthermore assume that $K_Y + f_* \Delta$ is $\mathbf{Q}$-Cartier, that $-E$ is $f$-ample, and that, in the case where $f (E)$ is not a point, the divisor $(K_Y + f_* \Delta) \vert_{f (E)}$ is ample.
\end{Assum}

\begin{Rem}
If $Y$ is $\mathbf{Q}$-factorial, then the condition that $-E$ is $f$-ample in Assumption \ref{Assum:II} is automatically satisfied, under Assumption \ref{Assum:I}.
(cf.\ Koll\'ar-Mori \cite{KoMo}, Lemma 2.62)
\end{Rem}

\begin{Defn}
Under Assumptions \ref{Assum:I} and \ref{Assum:II}.
We define the number $\lambda$ by the equation $K_X + \Delta + E = f^* (K_Y + f_* \Delta) + (1+\lambda) E$.
Then $1 + \lambda < 0$, because $K_X + \Delta + E$ is nup.
\end{Defn}

\begin{Prop}\label{Prop:II}
Under Assumptions \ref{Assum:I} and \ref{Assum:II}.
The divisor $K_X + \Delta + E$ is big.
\end{Prop}

\begin{proof}
Assume that $K_X + \Delta + (1-\epsilon) E = f^* (K_Y + f_* \Delta) + (1 + \lambda - \epsilon) E$ is not big for every small number $\epsilon >0$.
Thus its self intersection number is zero for every $\epsilon$ from Proposition \ref{Prop:I}.
Therefore $(-E)^{\dim E - \dim f (E)} \cdot ( f^*(K_Y + f_* \Delta)^{\dim f (E)} \cdot E ) = 0$.
This contradicts to the $f$-ampleness of $-E$.
Consequently $K_X + \Delta + (1-\epsilon) E$ is big for every small number $\epsilon > 0$ and so is $K_X + \Delta + (1-\epsilon) E + \epsilon E$.
\end{proof}

\begin{Prop}\label{Prop:III}
Under Assumptions \ref{Assum:I} and \ref{Assum:II}.
The divisor $(K_X + \Delta + E) \vert_E $ is ample.
\end{Prop}

\begin{proof}
The divisor $f^*(K_Y + f_* \Delta) \vert_E - \epsilon E \vert_E$ on $E$ is ample for every small number $\epsilon >0$ (cf.\ \cite{KoMo}, Proposition 1.45).
We also recall that $f^* (K_Y + f_* \Delta) + (1 + \lambda -\epsilon)E = K_X + \Delta + (1-\epsilon)E$ is nef by Proposition \ref{Prop:I}.
Thus $(K_X + \Delta + E) \vert_E = (f^* (K_Y + f_* \Delta) + (1 + \lambda) E) \vert_E$ is ample, from the inequality $-\epsilon > 1+\lambda > 1+\lambda-\epsilon$.
\end{proof}

We state the main result of this section:

\begin{Thm}\label{Thm:MT}
Under Assumptions \ref{Assum:I} and \ref{Assum:II}.
The divisor $K_X + \Delta + E$ is ample if and only if $((K_X + \Delta + E) \vert_{\Gamma})^{\dim \Gamma} > 0$ for every minimal log canonical (i.e.\ minimal non-klt) center $\Gamma$ with respect to the pair $(X, \Delta + E)$ such that $\Gamma \cap E = \emptyset$.
\end{Thm}

For proof, we cite the following ampleness result:

\begin{Prop}[\cite{Fu04}]\label{Prop:Fu04}
Let $(M, S)$ be a divisorial log terminal pair which is not Kawamata log terminal such that $M$ is projective.
Assume that the log canonical divisor $K_M + S$ is nup and that $((K_M + S) \vert_{\Gamma})^{\dim \Gamma} > 0$ for every minimal log canonical (i.e.\ minimal non-klt) center $\Gamma$ with respect to the pair $(M, S)$.
Then $K_M + S$ is ample.
\end{Prop}

\begin{proof}[Proof of Theorem \ref{Thm:MT}]
The ``only if '' part is trivial.
So we prove the ``if'' part.

For every minimal log canonical center $\Gamma$ with respect to $(X, \Delta + E)$ such that $\Gamma \cap E \ne \emptyset$, we have that $\Gamma \subset E$ from Ambro (\cite{Am98}, Proposition 3.3) because $E$ is a log canonical center with respect to $(X, \Delta + E)$.

Thus $((K_X + \Delta + E) \vert_{\Gamma})^{\dim \Gamma} > 0$ for every minimal log canonical center $\Gamma$ with respect to the pair $(X, \Delta + E)$ by Proposition \ref{Prop:III}.

Consequently Proposition \ref{Prop:Fu04} implies that $K_X + \Delta + E$ is ample.

\end{proof}

\begin{Exm}
Let $\mathbf{P}^n$ $(n \geq 3)$ be a projective space with homogeneous coordinate $( x_0 : x_1 : x_2 : \cdots : x_n )$ and hyperplane $G$.
We consider the hypersurface $Y$ ($\subset \mathbf{P}^n$) defined by the irreducible homogeneous equation $x_1^l + x_2^l + x_3^l + \cdots + x_m^l = 0$ $(3 \leq m \leq n)$ $(l \geq n+1)$.
We note that $Y$ is normal and that $K_Y = (-(n+1)G + lG) \vert_Y = (l-(n+1))G \vert_Y$ is Cartier.
Blow up $\mathbf{P}^n$ at the subspace $\{ x_1 = x_2 = \cdots = x_m = 0 \}$ and obtain the morphism $\phi: \mathbf{P}' \to \mathbf{P}^n$ and the exceptional divisor $F$.
Let $X$ be the strict transform of $Y$ by $\phi$.
We note that $X$ is nonsingular.
We have $K_{\mathbf{P}'} = \phi^* (-(n+1)G) + (m-1)F$.
Thus $K_X = (K_{\mathbf{P}'} + X) \vert_X = ( \phi^* (-(n+1)G) + (m-1)F + \phi^* (lG) - lF ) \vert_X = (\phi^* (l-(n+1))G -(l-(m-1))F) \vert_X $.
Then $K_X + (n-m+1)(\phi^* G) \vert_X + F \vert_X = (l-m) (\phi^* G - F) \vert_X$ is nef, because the linear system $\vert \phi^*G - F \vert$ is base point free.
Consequently $K_X + (\phi^* (n-m+2)G) \vert_X + F \vert_X =  ((l-m) \phi^* G - (l-m)F + \phi^* G) \vert_X$ is nup because $-F$ is $\phi$-ample. 

Let $H$ be the restriction of a general member of $\vert (n-m+2) G \vert$ to $Y$.
We put $f := \phi \vert_X$ and $\Delta := f^* H$.
Then $E := \Exc (f) = F \vert_X$ is a smooth prime divisor and $-E = -F \vert_X$ is $f$-ample.
We note that $K_X + \Delta + E$ is nup and that $(K_Y + f_* \Delta) \vert_{f(E)} = (l-(n+1)+(n-m+2)) G \vert_{f(E)}$ is ample.
When $\Delta \cap E = \emptyset$ (i.e.\ $n=m$), the divisor $(K_X + \Delta + E) \vert_{\Delta} = (l-(n+1)+2) \phi^* G \vert_{\Delta}$ is ample.
Lastly Theorem \ref{Thm:MT} implies that $K_X + \Delta + E$ is ample.

\end{Exm}

\section*{Appendix A: A survey of the Demailly-Hacon-P\u aun extension Theorem \cite{DHP}}

In this appendix, we survey the celebrated extension theorem due to Demailly-Hacon-P\u aun:

\begin{PropA}[\cite{DHP}]\label{PropA:DHP}
Let $(M, \Delta + S)$ be a projective purely log terminal pair with a prime divisor $S$ such that $ \lfloor \Delta + S \rfloor = S $.
Assume that the log canonical divisor $K_M + \Delta + S$ is nef and that there exists an effective $\mathbf{Q}$-divisor $D$ which is $\mathbf{Q}$-linearly equivalent to $K_M + \Delta + S$ with $S \subset \Supp (D) \subset \Supp (\Delta + S)$ .
Then the restriction map
$$ H^0(X, {\cal O}_X (m(K_M + \Delta + S))) \to H^0 (S, {\cal O}_S (m(K_M + \Delta + S))) $$
is surjective for all sufficiently large and divisible integers $m$.
\end{PropA}

Let $(X,B)$ be a projective Kawamata log terminal pair whose log canonical divisor $K_X + B$ is nef.

\begin{ConjA}[Log Abundance Conjecture]\label{ConjA:LA}
The (nef) log canonical divisor $K_X + B$ is semiample.
\end{ConjA}

\begin{Subconj}\label{Subconj:LA}
There exists an effective divisor $S$ on $X$ such that $(X, B + S)$ is purely log terminal and that $S$ is linearly equivalent to some multiple of $K_X + B$.
\end{Subconj}

\begin{PropA}\label{PropA:LA}
Log Abundance Conjecture A.\ \ref{ConjA:LA} implies Subconjecture \ref{Subconj:LA}.
\end{PropA}

\begin{proof}
If the logarithmic Kodaira dimension $\kappa (X, K_X + B) = 0$, then we are done (letting $S = 0$).
So we may assume that $\kappa (X, K_X + B) \geq 1$.

For a sufficiently large and divisible integer $l$, the linear system $\vert l(K_X+B) \vert$ is base point free and gives the algebraic fiber space $f: X \to T$.
Then $l(K_X + B)$ is linearly equivalent to $f^* H$ for some hyperplane section $H$ of $T$.
Consider a log resolution $\pi : Y \to X$ of $(X,B)$ such that the morphism $\pi$ is projective, that the exceptional locus $\Exc (\pi)$ is divisorial and that the locus $\Exc (\pi) \cup \Supp (\pi^* B)$ is with only simple normal crossings.
For a general member $S'$ of the linear system $\vert \pi ^* f^* H \vert$, the divisor $S' = \sum_{i \geq 1} S'_i$ is a disjoint union of a finite number of smooth prime divisors $S'_i$.
Thus the divisor $S:= \pi (S')$ satisfies the required condition.
\end{proof}

We consider the converse statement for Proposition A.\ \ref{PropA:LA}.

\begin{Clai}\label{Clai}
Under Subconjecture \ref{Subconj:LA}.
If $S \ne 0$ and $S_i$ is an irreducible component of $S$, then we have the
following properties:

(1) the prime divisor $S_i$ is a connected component of $S$.

(2) the pair $(S_i, K_{S_i} + \Diff ( B + S - S_i))$ is Kawamata log terminal.

(3) the log canonical divisor $K_{S_i} + \Diff ( B + S - S_i)$ is nef.

(4) the prime divisor $S_i$ is $\mathbf{Q}$-Cartier and nef.

(5) the restriction map
$$H^0 (X, {\cal O}_X (l(K_X + B + S))) \to H^0 (S_i, {\cal O}_{S_i} (l(K_X + B + S)))$$
is surjective for all sufficiently large and divisible integers $l$.
\end{Clai}

\begin{proof}
(1) and (2) are the elementary facts of purely log terminal pairs.
(3) is trivial.

Because $S$ is $\mathbf{Q}$-Cartier and nef, we have (4) from the fact that $S = \sum_j S_j$ is a disjoint union of prime divisors $S_j$.

Thus $S - S_i = \sum_{j \ne i} S_j$ is a nef $\mathbf{Q}$-Cartier divisor and $\Supp (S - S_i) \cap \Supp (S_i) = \emptyset$.
For a sufficiently small rational number $\epsilon > 0$, the pair $(X, B + \epsilon (S - S_i) + S_i)$ is purely log terminal and $\lfloor B + \epsilon (S - S_i) + S_i \rfloor = S_i$.
We note that there exists an effective $\mathbf{Q}$-divisor $D$ which is $\mathbf{Q}$-linearly equivalent to $K_X + B + \epsilon (S - S_i) + S_i$ such that $\Supp D = \Supp S$.
Because $S_i \subset \Supp S \subset \Supp (B + S) = \Supp (B + \epsilon (S - S_i) + S_i)$, we have that $S_i \subset \Supp D \subset \Supp (B + \epsilon (S - S_i) + S_i)$.
So we get the following commutative diagram from Proposition A.\ \ref{PropA:DHP} (\cite{DHP}):
$$
\begin{CD}
H^0 (X, {\cal O} (l(K_X + B + \epsilon (S - S_i) + S_i))) @>\text{rest.}>\text{surj.}> H^0 (S_i, {\cal O} (l(K_X + B + \epsilon (S - S_i) + S_i))) \to 0 \\
@VVV @VV\text{identity}V \\
H^0 (X, {\cal O} (l(K_X + B + S))) @>\text{rest.}>> H^0 (S_i, {\cal O} (l(K_X + B + S)))
\end{CD}
$$
\end{proof}

\begin{ThmA}\label{ThmA:LA}
Subconjecture \ref{Subconj:LA} in dimension $\le \dim X$ implies Log Abundance Conjecture A.\ \ref{ConjA:LA}.
\end{ThmA}

\begin{proof}
If $S=0$, we are done.
So we may assume that $S = \sum S_i \ne 0$, where $S_i$ are distinct prime divisors.
We follow the notation in Claim \ref{Clai}.
By induction on dimension, the log canonical divisor $K_{S_i} + \Diff (B + S - S_i)$ is semiample.
Therefore Claim \ref{Clai} (5) implies that the base locus $\Bs \vert l (K_X + B + S) \vert$ is disjoint from $S_i$ for a sufficiently large and divisible integer $l$.
Thus $\Bs \vert l (K_X + B + S) \vert$ is disjoint from $\sum S_i = S$.
From the assumption that $S$ is $\mathbf{Q}$-linearly equivalent to some multiple of $K_X + B$, the log canonical divisor $K_X + B$ is semiample.
\end{proof}

\begin{ConjA}[Smooth Abundance Conjecture]\label{ConjA:SA}
Assume that $X$ is smooth and $B=0$.
The (nef) canonical divisor $K_X$ is semiample.
\end{ConjA}

\begin{Subconj}\label{Subconj:SA}
Assume that $X$ is smooth and $B=0$.
There exists an effective divisor $S$ such that $(X, S)$ is log smooth and purely log terminal and that $S$ is linearly equivalent to some multiple of $K_X$ .
\end{Subconj}

By the same argument as in the proofs of Proposition A.\ \ref{PropA:LA} and Theorem A.\ \ref{ThmA:LA}, we have the following two results.

\begin{PropA}
Smooth Abundance Conjecture A.\ \ref{ConjA:SA} implies Subconjecture \ref{Subconj:SA}.
\end{PropA}

\begin{ThmA}
Subconjecture \ref{Subconj:SA} in dimension $\le \dim X$ implies Smooth Abundance Conjecture A.\ \ref{ConjA:SA}.
\end{ThmA}

\section*{Appendix B: Strong log bigness}

Let $X$ be a projective variety over the field $\mathbf{C}$ of complex numbers and $\Delta$ an effective $\mathbf{Q}$-divisor on $X$ where the pair $(X, \Delta)$ is dlt (i.e.\ divisorial log terminal).

\begin{DefnA}\label{DefnA:SLB}
A $\mathbf{Q}$-Cartier $\mathbf{Q}$-divisor $D$ is {\it strongly log big} on $(X, \Delta)$ if, for some integer $m >0$, the following three conditions are satisfied:

(i) the $\mathbf{Q}$-Cartier $\mathbf{Q}$-divisor $mD$ is a Cartier divisor,

(ii) the base locus $\Bs \vert mD \vert$ does not contain any generic point of the log canonical centers of $(X, \Delta)$,

(iii) the rational map $\phi := \Phi_{\vert mD \vert}$ is birational to its image and, furthermore, is isomorphic onto its image in some neighborhood of every generic point of the log canonical centers of $(X, \Delta)$.
\end{DefnA}

\begin{RemA}\label{RemA:BCL}
Boucksom-Cacciola-Lopez (\cite{BCL}) proved that, for a big divisor $D$, the strong log bigness of $D$ is equivalent to the condition that the augmented base locus $\mathbf{B}_{+} (D)$ does not contain any generic point of the log canonical centers.
\end{RemA}

\begin{ThmA}[\cite{BCL} and \cite{BH}-\cite{Ca}]
If the log canonical divisor $K_X + \Delta$ is strongly log big on the dlt pair $(X, \Delta)$, then the log canonical ring $\bigoplus _{k \ge 0} H^0 (X, {\cal O}_X (\lfloor k (K_X + \Delta) \rfloor))$ is finitely generated over the field $\mathbf{C}$.
\end{ThmA}

From Remark A.\ \ref{RemA:BCL}, the theorem above is a reduction of Birkar-Hu (\cite{BH}) or Cacciola (\cite{Ca}).
But we give a straightforward proof to the theorem.

\begin{proof}
We follow the notation in Definition A.\ \ref{DefnA:SLB} for the $\mathbf{Q}$-Cartier $\mathbf{Q}$-divisor $K_X + \Delta$.
From the assumption and the divisorial log terminal theorem (Szab\'o \cite{Sz}), there exists some nonempty Zariski-open subset $U$ of $X$ with the following properties:

(i) $U$ contains all the generic points of log canonical centers of $(X, \Delta)$,

(ii) $\Bs \vert m(K_X + \Delta) \vert \cap U = \emptyset$,

(iii) the rational map $\phi \vert_U$ is isomorphic onto its image,

(iv) the pair $(U, \Delta \vert_U)$ is a nonsingular variety $U$ with a reduced simply normal crossing divisor $\Delta \vert_U$ on $U$.

We set $Y := [\text{the image of the rational map $\phi$}]$.

From the resolution lemma (\cite{Sz}) due to Szab\'o, there exists a log resolution $\mu \colon X_1 \to X$ of the pair $(X, \Delta)$ such that $\mu \vert_{\mu^{-1} (U)}$ is isomorphic and that $\Exc (\mu)$ is divisorial.

Here the exceptional locus $\Exc$ denotes the locus where the morphism is not isomorphic.

From the Hironaka resolution theorem, by the repetitions of blowups along smooth subvarieties included in the singular locus of $Y$, we have a resolution $\nu \colon Y_1 \to Y$ of singularities such that $\nu \vert_{\nu^{-1} (\phi (U))}$ is isomorphic and that there exists some $\nu$-antiample effective divisor whose support coincides with the exceptional locus $\Exc (\nu)$.

We consider the rational map $\phi_1 := \nu^{-1} \phi \mu$.
Then we obtain the commutative diagram:

\begin{equation}
\begin{CD}
X_1 @>\text{$\phi_1$}>\text{rational map}> Y_1 \\
@V\text{$\mu$}VV @VV\text{$\nu$}V \\
X @>\text{rational map}>\text{$\phi$}> Y
\end{CD}
\end{equation}

We take the elimination of indeterminacy for the rational map $\phi_1$:

$$X_1 \stackrel{\mu_1}{\longleftarrow} X_2 \stackrel{\phi_2}{\longrightarrow} Y_1.$$

Note that the morphism $\mu_1 \vert_{\mu_1^{-1} (\mu^{-1}(U))}$ is isomorphic.

Because the variety $X_1$ ($Y_1$, respectively) is $\mathbf{Q}$-factorial, there exists some $\mu_1$-antiample ($\phi_2$-antiample, respectively) effective divisor whose support coincides with $\Exc (\mu_1)$ ($\Exc (\phi_2)$, respectively).

We put $\mu_2 := \mu \mu_1$.
Then we have the commutative diagram:

\begin{equation}
\begin{CD}
X_2 @>\text{$\phi_2$}>> Y_1 \\
@V\text{$\mu_2$}VV @VV\text{$\nu$}V \\
X @>\text{rational map}>\text{$\phi$}> Y
\end{CD}
\end{equation}

We have the relation
$$\vert {\mu_2}^* (m(K_X + \Delta)) \vert = \vert {\phi_2}^* \nu^* A_0 \vert + B_0$$
between complete linear systems where $A_0$ is a hyperplane section of $Y$ and $B_0$ is an effective divisor on $X_2$ with the property that $\Supp (B_0) \cap {\mu_2}^{-1}(U) = \emptyset$.

We consider the $\mathbf{Q}$-divisor $A := \frac{1}{m} A_0$.

We set $\Gamma := \Supp ({\mu_2}^{-1}_{*} \Delta) \cup \Exc (\mu_2) \cup \Exc (\nu \phi_2) \cup \Supp (B_0)$, which is purely codimension $1$ in $X_2$.

Consider the Zariski open subset $V := ( X_2 \setminus \Gamma ) \cup {\mu_2}^{-1} (U)$.

We note that $V \cap \Gamma = \Supp ({\mu_2}^{-1}_* \Delta) \cap {\mu_2}^{-1} (U)$ and that $X_2 \setminus V \subset \Gamma$.

From the resolution lemma (\cite{Sz}) due to Szab\'o, we have a projective morphism $\mu_3 \colon \tilde{X} \to X_2$ which satisfies the following four conditions:

(a) $\mu_3$ is a composition of blowups along smooth subvarieties,

(b) $\mu_3 \vert_{{\mu_3}^{-1} (V)}$ is isomorphic,

(c) $\tilde{X}$ is nonsingular,

(d) ${\mu_3}^{-1} (\Gamma)$ is a divisor with only simple normal crossings.

Putting $\tilde{\nu} := \nu \phi_2 \mu_3$ and $\tilde{\mu} := \mu_2 \mu_3$, we have the diagram
\begin{equation}
\begin{CD}
\tilde{X} @>\text{$\tilde{\nu}$}>> Y \\
@V\text{$\tilde{\mu}$}VV \\
X
\end{CD}
\end{equation}
and have the property that the loci $\Exc (\tilde{\mu})$ and $\Exc (\tilde{\nu})$ are divisorial.
We define the $\mathbf{Q}$-divisors $\tilde{E}$, $\tilde{F}$ and $\tilde{B}$ by the following relations:

(i) $K_{\tilde{X}} + {\tilde{\mu}}^{-1}_* \Delta + \tilde{E} = {\tilde{\mu}}^* (K_X + \Delta) + \tilde{F}$,

(ii) ${\tilde{\mu}}^* (K_X + \Delta) = \tilde{\nu}^* A + \tilde{B}$ (i.e.\ $\tilde{B} = \frac{1}{m} \mu_3^* B_0$),

(iii) $\tilde{E}, \tilde{F}, \tilde{B} \geq 0$,

(iv) $\tilde{E}$ and $\tilde{F}$ have no common irreducible component.

Then we have the properties that
$\Supp ({\tilde{\mu}}^{-1}_* \Delta + \tilde{E} + \tilde{F} + \tilde{B}) \cup \Exc (\tilde{\mu}) \cup \Exc (\tilde{\nu})$ is a reduced divisor with only simple normal crossings, that $\Supp (\tilde{E}$ + $\tilde{F}$ + $\tilde{B})$ is disjoint from ${\tilde{\mu}}^{-1} (U)$ and that $\lfloor \tilde{E} \rfloor = 0$.

There exists some $\nu$-antiample ($\phi_2$-antiample, $\mu_3$-antiample, respectively) effective divisor whose support is $\Exc (\nu)$ ($\Exc (\phi_2)$, $\Exc (\mu_3)$, respectively).
Thus the $\mathbf{Q}$-divisors
$$\nu^*A -S_1,$$
$${\phi_2}^* (\nu^* A - S_1) - S_2$$
and
$$\tilde{A} := {\mu_3}^* ({\phi_2}^* (\nu^* A - S_1) - S_2) - S_3$$
are ample for some effective $\mathbf{Q}$-divisors $S_1$, $S_2$, $S_3$ with the property that $\Supp(S_1) = \Exc(\nu)$, $\Supp(S_2) = \Exc(\phi_2)$, $\Supp(S_3) = \Exc(\mu_3)$.
We write ${\tilde{\nu}}^* A = \tilde{A} + S$ where $S := {\mu_3}^* {\phi_2}^* S_1 + {\mu_3}^* S_2 + S_3 \geqq 0$ and note that $\Supp (S) = \Exc (\tilde{\nu})$.
Then
\begin{align}
K_{\tilde{X}} + \tilde{\mu}^{-1}_* \Delta + \tilde{E} &= {\tilde{\nu}}^* A + \tilde{B} + \tilde{F} \\
&= \tilde{A} + S + \tilde{B} + \tilde{F} \\
&= \tilde{A} + F
\end{align}
where $F := S + \tilde{B} + \tilde{F}$.

Here $\Supp(F) \cap \tilde{\mu}^{-1} (U) = \emptyset$.
Thus $\Supp (F)$ does not include any log canonical center of the smooth pair $(\tilde{X}, \tilde{\mu}^{-1}_* \Delta + \tilde{E} )$.
For a sufficiently small rational number $\delta >0$, the $\mathbf{Q}$-divisor $\delta (\tilde{\mu}^{-1}_* \Delta) + \tilde{A}$ is ample.
Therefore, for a sufficiently large and divisible integer $l > 0$, the divisor $l (\delta (\tilde{\mu}^{-1}_* \Delta) + \tilde{A})$ is very ample and linearly equivalent to some prime divisor $H$ such that $\Supp ( \tilde{\mu}^{-1}_* \Delta + \tilde{E} + F + H )$ is with only simple normal crossings and that $H$ does not include any log canonical center of the smooth pair $(\tilde{X}, \tilde{\mu}^{-1}_* \Delta + \tilde{E})$.
We have the following relation and the klt (i.e. Kawamata log terminal) pair $(\tilde{X}, (1-\epsilon \delta) \tilde{\mu}^{-1}_* \Delta + \tilde{E} + \epsilon F + \frac{\epsilon}{l} H )$ for a sufficiently small rational number $\epsilon >0$:
\begin{align}
(1 + \epsilon )(K_{\tilde{X}} + \tilde{\mu}^{-1}_* \Delta + \tilde{E}) &\sim_{\mathbf{Q}} K_{\tilde{X}} + \tilde{\mu}^{-1}_* \Delta + \tilde{E} + \epsilon \tilde{A} + \epsilon F \\
&\sim_{\mathbf{Q}} K_{\tilde{X}} + \tilde{\mu}^{-1}_* \Delta + \tilde{E} + \epsilon (\frac{1}{l} H - \delta \tilde{\mu}^{-1}_* \Delta) + \epsilon F \\
&\sim_{\mathbf{Q}} K_{\tilde{X}} + (1 - \epsilon \delta) \tilde{\mu}^{-1}_* \Delta + \tilde{E} + \epsilon F + \frac{\epsilon}{l} H
\end{align}

From the Birkar-Cascini-Hacon-McKernan theorem (\cite{BCHM}), the log canonical ring $\bigoplus _{m \ge 0} H^0 (\tilde{X}, {\cal O}_{\tilde{X}} (\lfloor m (K_{\tilde{X}} + (1 - \epsilon \delta) \tilde{\mu}^{-1}_* \Delta + \tilde{E} + \epsilon F + \frac{\epsilon}{l} H) \rfloor))$ for a klt pair is finitely generated.

Consequently the equivalence between the finite generation of the log canonical ring and that of some truncation of this ring implies the assertion. 

\end{proof}

\noindent{\bf Disclosure}

\noindent
The content of Remark \ref{Hist} (History)(see \cite{Fu14}) was presented in the short communications at ICM 2014 (Seoul) on August 16 in the year 2014.
Remark \ref{Hist} corrects a chronological typo and a chronological mistake in Fukuda (\cite{Fu14}).

\noindent{\bf Conflict of Interests}

\noindent
The author declares that there is no conflict of interests regarding the publication of this article.

\noindent{\bf Acknowledgments}

\noindent
In this series of research, the author was supported by the research grant of Gifu Shotoku Gakuen University in the years 2011 and 2014.
The author would like to thank the referee who carefully read the manuscript and gave the advice to improve the explanation.

\bigskip
Faculty of Education, Gifu Shotoku Gakuen University

Yanaizu-cho, Gifu City, Gifu 501-6194, Japan

fukuda@ha.shotoku.ac.jp

the current email address: fukuda@gifu.shotoku.ac.jp

\end{document}